\documentclass[a4paper, twoside,12pt]{article}
\usepackage{fancyhdr}
\usepackage{fnpos}
 \usepackage[english]{babel}        
\usepackage[T1]{fontenc}          
\usepackage{graphicx}             
\usepackage{makeidx}
\usepackage{fancybox}
\usepackage{framed}
\usepackage{fancyhdr}
 \usepackage{pstricks,pst-plot,pstricks-add}
\usepackage[margin=1in]{geometry}
\usepackage{graphicx}
\usepackage{titlesec}
\usepackage{amsmath}
\usepackage{amsfonts}   
\usepackage{amssymb}    
\usepackage{amsthm}
\usepackage{dsfont}
\usepackage{mathtools}
\usepackage{verbatim}   
\usepackage{color}
\usepackage{listings}
\usepackage{graphicx}
\usepackage{amsmath}
\usepackage{amsmath,amssymb,color,inputenc,euscript,graphicx,psfrag}

\newtheorem{thm}{Theorem}[section]
\newtheorem{cor}[thm]{Corollary}
\newtheorem{lem}[thm]{Lemma}
\newtheorem{prop}[thm]{Proposition}

\newtheorem{rem}[thm]{Remark}
\numberwithin{equation}{section}

\renewcommand{\thefootnote}

\newcommand\co{\operatorname{co}}

 {\markboth{\sl   }{\sl  }

\author {Amel Hammi \& B\'echir Amri   }

\title{  Dunkl--Schr\"{o}dinger operators}

\date{ }

\begin{document}
 \maketitle
\begin{center}
   Universit\'{e} Tunis El Manar, Facult\'{e} des sciences de Tunis,\\ Laboratoire d'Analyse Math\'{e}matique
       et Applications,\\ LR11ES11, 2092 El Manar I, Tunisie.\\
     \textbf{ e-mail:} bechiramri69@gmail.com, hammiamel097@gmail.com
\end{center}
 \begin{abstract}
  In this paper, we consider the Schr\"{o}dinger operators $L_k=-\Delta_k+V$, where $\Delta_k$ is the Dunkl-Laplace operator and $V$ is a  non-negative potential  on $R^d$.
We establish that $L_k $ is essentially self-adjoint on $C_0^\infty$. In particular,
we develop a bounded $H^\infty$-calculus on $L^p$ spaces for  the Dunkl harmonic oscillator operator.
 \end{abstract}
\section{Introduction  }
During the recent decades the ordinary Schr\"{o}dinger operators   $-\Delta+V$   have been generalized  in  a domains  where a   family of  a Laplace type operators are  given,  for example, in Heisenberg group \cite{LL}, nilpotent Lie groups \cite{HL}, and
spaces of homogeneous type \cite{BB}. In a similar way,  in the area of harmonic analysis, functional calculus for  self-adjoint operators and a number of important applications are developed \cite{DR, DM, Mc1}.
 \par In this paper we consider Schr\"{o}dinger operators $L_k=-\Delta_k+V$ associated to the Dunkl Laplace operator on $\mathbb{R}^d$ given by
$\Delta_k=\sum_{j=1}^{^d}T_j^2$
 where $T_j$ are  a family of differential-difference operators
associated to a finite reflection group, which are called Dunkl
operators.  It is well known that Dunkl theory provides a generalization  of the   Fourier Analysis. There are many classical results  in  Fourier analysis that are extended  to Dunkl setting and the present work  fits within this framework.
Arise from the work of Simon \cite{SS}, we study the problem of essential selfadjointness of $L_k$ and the correspondent heat   semi group $e^{-tL_k}$.
We investigate the spectral theory,   complex analysis and theory    of  holomorphic functional calculus for  the  generator of an holomorphic semi group, as present in \cite{DR, DM, Mc1} to develop an $L^p$- boundedness of holomorphic functional calculus for  the Dunkl harmonic oscillator
$H_k=-\Delta_k+|x|^2$, where the  use  of the heat kernel  being the    most powerful tools.
\par The paper is outline as follows. In the next Section we give Backgrounds form Dunkl's theory. The Sections III and IV are devoted to study the  essential  selfadjointness  of the operators $\Delta_k$ and $L_k$. The Section V treat  the $L^p$- boundedness of holomorphic functional calculus for $H_k$.
 \section{Basics of the Dunkl  theory}
For details, we refer to \cite{D1, D2, J1, R2} and the references cited there.
 \par
$\mathbb{R}^d$ is equipped with a scalar product $\langle x,y\rangle=\sum_{j=1}^dx_jy_j$
 which induces the Euclidean norm $|x|=\langle x, x\rangle^{1/2}$.
\par Let $G\!\subset\!\text{O}(\mathbb{R}^d)$
be a finite reflection group associated to a reduced root system $R$
and $k:R\rightarrow[0,+\infty)$ be  a $G$--invariant function
(called multiplicity function).
Let $R^+$ be a positive root subsystem. The Dunkl operators \,$T_\xi$ on $\mathbb{R}^d$ are
the following $k$--de\-for\-ma\-tions of directional derivatives $\partial_\xi$
by difference operators\,:
$$\textstyle
T_\xi f(x)=\partial_\xi f(x)
+\sum_{\,\alpha\in R^+}\!k(\alpha)\,\langle\alpha,\xi\rangle\,
\frac{f(x)-f(\sigma_\alpha.\,x)}{\langle\alpha,\,x\rangle}\,,
$$
where  $\sigma_\alpha $ denotes the reflection
with respect to the hyperplane orthogonal to $\alpha$.
The Dunkl operators are antisymmetric
with respect to the measure $w_k(x)\,dx$
with density
$$\textstyle
w_k(x)=\,\prod_{\,\alpha\in R^+}|\,\langle\alpha,x\rangle\,|^{\,2\,k(\alpha)}\,.
$$
The operators $\partial_\xi$ and $T_\xi$
are intertwined by a Laplace--type operator
\begin{eqnarray*}\label{vk}
V_k\hspace{-.25mm}f(x)\,
=\int_{\mathbb{R}^d}\hspace{-1mm}f(y)\,d\nu_x(y),
\end{eqnarray*}
associated to a family of compactly supported probability measures
\,$\{\,\nu_x\,|\,x\!\in\!\mathbb{R}^d\hspace{.25mm}\}$\,.
Specifically, \,$\nu_x$ is supported in the the convex hull $\co(G.x)\,.$
\par For every $y\!\in\!\mathbb{C}^d$\!,
the simultaneous eigenfunction problem
\begin{equation*}
T_\xi f=\langle y,\xi\rangle\,f
\qquad\forall\;\xi\!\in\!\mathbb{R}^d
\end{equation*}
has a unique solution $f(x)\!=\!E_k(x,y)$
such that $E_k(0,y)\!=\!1$, called the Dunkl kernel and is given by
\begin{equation}\label{EV}
E_k(x,y)\,
=\,V(e^{\,\langle\lambda,\,.\,\rangle})(x)\,
=\int_{\mathbb{R}^d}\hspace{-1mm}e^{\,\langle z,y\rangle}\,d\nu_x(z)
\qquad\forall\;x\!\in\!\mathbb{R}^d.
\end{equation}
Furthermore this kernel has a holomorphic extension to $\mathbb{C}^d\times \mathbb{C}^d $
 and the following estimate hold\,: for  $ \;x, \;y\!\in\!\mathbb{C}^d,$
\begin{itemize}
\item[(ii)] $E_k(x,y)=E_k(y,x)$,
\item[(iii)] $E_k(\lambda x,y)=E_k(x,\lambda  y)$, for $\lambda\in \mathbb{C}$
\item[(iv)] $E_k(g. x,g.y)=E_k(x, y)$, for $g\in G$.
\end{itemize}
 In dimension $d\!=\!1$,
these functions can be expressed in terms of Bessel functions.
Specifically,
$$\textstyle
E_k(x,y)= \mathcal{J}_{k-\frac12}(xy)
+\frac{xy}{2\hspace{.25mm}k+1}\, \mathcal{J}_{k+\frac12}(xy)\,,
$$
where
$$\textstyle
 \mathcal{J}_\nu(z)\,=\;\Gamma(\nu\!+\!1)\;
{\displaystyle\sum\nolimits_{\,n=0}^{+\infty}}\;
\frac{(-1)^n}{n\hspace{.25mm}!\,\Gamma(\nu+n+1)}\;
\bigl(\frac z2)^{2n}
$$
are normalized Bessel functions.
\par The Dunkl transform  is defined on $L^1(\mathbb{R}^d\!,w_k(x)dx)$ by
$$
\mathcal{F}_kf(\xi)={\textstyle c_k}
\int_{\mathbb{R}^d}\!f(x)\,E_k(x,-i\,\xi)\,w_k(x)\,dx\,,
$$
where
$$
c_k\,=\int_{\mathbb{R}^d}\!e^{-\frac{|x|^2}2}\,w(x)\,dx\,.
$$
We list some known properties of this transform\,:
\begin{itemize}
\item[(i)]
The Dunkl transform is a topological automorphism
of the Schwartz space $\mathcal{S}(\mathbb{R}^d)$.
\item[(ii)]
(\textit{Plancherel Theorem\/})
The Dunkl transform extends to
an isometric automorphism of $L^2(\mathbb{R}^d\!,w_k(x)dx)$.
\item[(iii)]
(\textit{Inversion formula\/})
For every $f\!\in\!\mathcal{S}(\mathbb{R}^d)$,
and more generally for every $f\!\in\!L^1(\mathbb{R}^d\!,w_k(x)dx)$
such that $\mathcal{F}_kf\!\in\!L^1(\mathbb{R}^d\!,w_k(\xi)d\xi)$,
we have
$$
f(x)=\mathcal{F}_k^2\!f(-x)\qquad\forall\;x\!\in\!\mathbb{R}^d.
$$
\item[(iv)] if $f$ is a radial function in $L^1(\mathbb{R}^d\!,w_k(\xi)d\xi)$ such that  $f(x)=\widetilde{f}(|x|)$, then
$\mathcal{F}_k(f)$ is also radial and one has
\begin{equation}\label{rad}
    \mathcal{F}_k(f)(x)=b_k\int_0^\infty\widetilde{f}(s) \mathcal{J}_{\gamma_k+d/2-1}(s|x|)s^{2\gamma_k+d}ds.
\end{equation}
where $b_k= 2^{-(\gamma_k+d/2-1)}/\Gamma(\gamma_k+d/2)$.
\end{itemize}
\par Let $x\in \mathbb{R}^d$, the
Dunkl translation operator $\tau_x$ is given for $f\in
L^2_k(\mathbb{R}^d,\mathbb{C})$ by
\begin{eqnarray*}\label{dutr}
\mathcal{F}_k(\tau_x(f))(y)= \mathcal{F}_kf(y)\,E_k(x,iy), \quad
y\in\mathbb{R}^d.
\end{eqnarray*}
Using the Dunkl's intertwinig operator $V_k$, the operator  $\tau_x $ is related
to the usual translation by
$$\tau_x(f)(y)= (V_k)_x(V_k)_y((V_k)^{-1}(f)(x + y)).$$
In the case when $f(x)=\widetilde{f}(|x|)$ is a radial function in  $  \mathcal{S}(\mathbb{R}^d)$,  the Dunkl translation is represented by the following integral
\begin{eqnarray}\label{trad}
\tau_x(f)(y)=
\int_{\mathbb{R}^{n}}\widetilde{f}( \sqrt{|y|^2+|x|^2+2<y,\eta>}\;)\;d\nu_x(\eta).
 \end{eqnarray}
This formula shows that the Dunkl translation operators can   be extended to all radial functions $f$ in $L^p (\mathbb{R}^d,w_k(x)dx)$, $1\leq p\leq \infty$  and the following holds
\begin{equation}\label{trp}
  ||\tau_x(f)||_{p,k}\leq ||f||_{p,k },
  \end{equation}
where $\|.\|_{p,k}$ is the usual norm of $L^p (\mathbb{R}^d,w_k(x)dx)$.
\par We define  the Dunkl convolution product   for suitable functions $f$ and $g$ by
$$f*_kg(x)=\int_{\mathbb{R}^d} \tau_x(f)(-y)g(y)d\mu_k(y),\quad x\in\mathbb{R}^d$$
We note that it is commutative and satisfies the following property:
\begin{eqnarray}\label{conv}
 \mathcal{F}_k(f*_kg)=\mathcal{F}_k(f)\mathcal{F}_k(g), \quad \quad f,\;g\in L^2(\mathbb{R}^d,w_k(x)dx).
\end{eqnarray}
Moreover, the operator $ f \rightarrow f*_kg $ is bounded on $L^p (\mathbb{R}^d,w_k(x)dx)$
provide $g$ is a bounded radial function in $L^1(\mathbb{R}^d,w_k(x)dx)$. In particular we have the   the following Young's inequality:
\begin{equation}\label{Y}
 \|f*_kg\|_{p,k}\leq \|g\|_{1,k}\|f\|_{p,k}\;.
\end{equation}
 \section{ The  Dunkl Laplacian  operator}
In this section and in what follows, we often use the language of spectral theory for unbounded operators. Our main reference is \cite{RS}.
\par Let $(e_1,e_2,...,e_d)$ be an orthonormal basis of  $\mathbb{R}^d$. The Dunkl Laplacian operator is
defined by
 $$\Delta_k=\sum_{j=1}^dT_{j}^2,$$
 where $T_j=T_{e_j}$.  We consider   $-\Delta_k$ as a densely defined   operator on the Hilbert space $L^2(\mathbb{R}^d,w_k(x)dx)$  with domain $D(-\Delta_k)=\mathcal{S}(\mathbb{R}^d)$,  the Schwartz space of rapidly decreasing functions.
By means of the Dunkl transform one can prove that, for $f,g\in \mathcal{S}(\mathbb{R}^d)$,
\begin{eqnarray*}
 -\Delta_k(f)(x)&=&\mathcal{F}_k^{-1}(|x|^2\mathcal{F}_k(f)),
 \\\langle -\Delta_k f, g \rangle&=& \langle  f,- \Delta_k g \rangle,
 \\\langle- \Delta_k f, f \rangle&=&\sum_{j=1}^d\|T_j(f)\|_{2,k}^2,
  \end{eqnarray*}
   which show that  $-\Delta_k$ is a densely defined,  symmetric  and positive operator on $L^2(\mathbb{R}^d,w_k(x)dx)$. Thus, the Friedrichs extension theorem tells us that there is a positive self-adjoint extension of $-\Delta_k$.  Define the linear operator $A_k$ as extension of $-\Delta_k$  by
  $$D(A_k)=H_k^2(\mathbb{R}^d)=\{f\in L^2(\mathbb{R}^d,w_k(x)dx);\; |x|^2\mathcal{F}_k(f)\in L^2(\mathbb{R}^d,w_k(x)dx)\}$$
  $$ A_k(f) =\mathcal{F}_k^{-1}(|x|^2\mathcal{F}_k(f)); \qquad f \in D(A_k).$$
Clearly $A_k$ is symmetric  and positive.
\begin{thm}\label{th1}
  The operator  $A_k$ is  self-adjoint and that is the unique positive self-adjoint  extension operator of $-\Delta_k$.
\end{thm}
\begin{proof}
 Recall first that the adjoint operator $A_k^*$ is  defined on the domain  $D(A_k^*)$
 consisting of the function $f\in  L^2(\mathbb{R}^d,w_k(x)dx)$ for which the functional
  $  g\mapsto \langle A_kg , f\rangle $ is   bounded   on  $D(A_k)$
and  by Riesz representation theorem there exists a unique $f^*\in L^2(\mathbb{R}^d,w_k(x)dx)$ such that $\langle A_k(g),f\rangle=\langle g,f^*  \rangle$, since $D(A_k)$ is dense in $L^2(\mathbb{R}^d,w_k(x)dx)$. We define $A_k^*$ on  $D(A_k^*)$ by
$A_k^*(f)=f^*$. Since $A_k$ is symmetric,  then $A_k$ is  self-adjoint   if and only if  $D(A_k^*)=D(A_k)$. Noting   that
$D(A_k)\subset D(A_k^*)$ is obvious. Let $f\in D(A_k^*)$, then   there exists a constant $C>0$ such that for all $g \in D(A_k)$ we have
\begin{equation}\label{b}
 \langle   A_kg,f \rangle\leq C\|g\|_{2,k}.
\end{equation}
For $r>0$ define the function  $g_r \in  L^2(\mathbb{R}^d,w_k(x)dx)$  by
$$\mathcal{F}_k(g_r)(x)=|x|^2\mathcal{F}_k(f)(x)\chi_{\{|x|<r\}} $$
where $\chi$ denotes the characteristic function. Clearly we have $g_r\in D(A_k)$  and in view of (\ref{b})
\begin{eqnarray*}
  |\langle f, A_kg_r\rangle|&=&|\langle\mathcal{F}_k(f),\mathcal{F}_k(A_kg_r)\rangle|=\int_{|x|<r}|x|^4|\mathcal{F}_k(f)(x)|^2w_k(x)dx
\\&\leq &C \left(\int_{|x|<r}|x|^4|\mathcal{F}_k(f)(x)|^2w_k(x)dx\right)^{1/2}.
\end{eqnarray*}
It yields that
$$\int_{|x|<r}|x|^4|\mathcal{F}_k(f)(x)|^2w_k(x)dx\leq C^2.$$
Therefore by letting $r\rightarrow \infty$   we deduce  that $f\in  D(A_k)$ and  conclude that $D(A_k^*) \subset D(A_k)$.
\par Let us now   prove that $-\Delta_k$    is essentially self-adjoint, this means  that $-\Delta_k$ admits an  unique self-adjoint extension and that is equal to $A_k$, since we have proved that $A_k$ is self-adjoint.
 From   the general theory of unbounded operators,  see for example the chapter VIII of  \cite{RS}, it suffices to prove
that    $(-\Delta_k \pm i)D(-\Delta_k)$  is  dense, which is equivalent to $$\Big((-\Delta_k \pm i)D(-\Delta_k)\Big)^\perp=\{0\}.$$
  In fact, let $g\in \Big((-\Delta_k \pm i)D(-\Delta_k)\Big)^\perp$. Then for any $f\in D(-\Delta_k)=\mathcal{S}(\mathbb{R}^d)$
  \begin{eqnarray*}
    0=\langle(-\Delta_k\pm i)f,g\rangle = \langle (|.|^2\pm i)\mathcal{F}_k(f),\mathcal{F}_k(g)\rangle=
   \langle (\mathcal{F}_k(f),(|.|^2\pm i)\mathcal{F}_k(g)\rangle
  \end{eqnarray*}
Since $\mathcal{F}_k(\mathcal{S}(\mathbb{R}^d))=\mathcal{S}(\mathbb{R}^d)$, this implies by density argument that  $\mathcal{F}_k(g)=0$
and so $g=0$, as desired.
\end{proof}
 \par
The operator $A_k$ is generator of strongly continuous one parameter semi group $(e^{-tA_k})_{t\geq 0}$
where the operator  $e^{-tA_k}$ is given by
$$e^{-tA_k}f=  \mathcal{F}_k^{-1} (e^{-t|.|^2} \mathcal{F}_k(f))$$
for all $t>0$ and $f\in L^2(\mathbb{R}^d,w_k(x)dx).$
  It follows that  $e^{-tA_k}$ is an integral operator given by
\begin{equation}\label{1}
   e^{-tA_k}f(x)==k_t*_kf =\int_{\mathbb{R}^d}K_t(x,y)f(y)w_k(y)dy
\end{equation}
where
\begin{equation}\label{kk}
 k_t(x)= \mathcal{F}_k^{-1}(e^{-t|.|^2})(x) = t^{-\gamma_k-d/2}e^{-|x|^2/t}
\end{equation}
and from  (\ref{EV}) and  (\ref{trad})
\begin{equation}\label{KK}
 K_t(x,y)=\tau_x(k_t)(-y)= t^{-\gamma_k-d/2}e^{- (|x|^2+|y|^2)/t}E(2x/t,y).
\end{equation}
Now   using H\"{o}lder's Inequality,   one can  state  the following
\begin{cor}\label{3.2}
  $e^{-tA_k}$ can be extended to a bounded operator from $L^p(\mathbb{R}^d,w_k(x)dx)$ to $L^\infty(\mathbb{R}^d,w_k(x)dx)$ , for $1\leq p\leq \infty$.
\end{cor}
\section{Dunkl Schr\"{o}dinger operator}
In this section, we use arguments analogous to those used in \cite{SS}.
\par Let $V$ be a nonnegative measurable function on $\mathbb{R}^d$ that is finite almost everywhere.
 In   the Hilbert space $L^2(\mathbb{R}^d,w_k(x)dx)$ we consider the  operator
 $$\mathcal{L}_k=A_k+V$$
  with domain   $D(\mathcal{L}_k)=D(A_k)\cap D(V)$ where
 $$D(A_k)=H_k^2(\mathbb{R}^d)=\{f\in L^2(\mathbb{R}^d,w_k(x)dx);\; |x|^2\mathcal{F}_k(f)\in L^2(\mathbb{R}^d,w_k(x)dx)\}$$
 and
 $$D(V)=\{ f\in L^2(\mathbb{R}^d,w_k(x)dx);\; Vf\in L^2(\mathbb{R}^d,w_k(x)dx)\}.$$
 We call this operator the Dunkl Schr\"{o}dinger operator. We have already proven that $A_k$ is positive self-adjoint operator, we should add here that
 the multiplication operator $f\rightarrow Vf$ is a positive self-adjoint, see \cite{RS}, VIII.3, Proposition 1.
 The important fact that we shall use comes from the    theory of  the  quadratic form. Define the form $q_k$ by
 \begin{eqnarray*}
 D(q_k)&=&\{f\in L^2(\mathbb{R}^d,w_k(x)dx);\;\left(\sum_{j=1}^n|T_jf|^2\right)^{1/2},   V^{1/2}f \in  L^2(\mathbb{R}^d,w_k(x)dx) \}\\
 q_k(f)&=&\sum_{j=1}^n\|T_jf\|_{2,k}^2+\|V^{1/2}f\|^2_{2,k}
\end{eqnarray*}
Clearly $C_0^\infty\subset D(q_k)$, so $q_k$ is densely defined.
\begin{lem}
  $q_k$   is closed, that is if $(\varphi_n)_n\in D(q_k)$, $\|\varphi_n-\varphi\|_{2,k}\rightarrow0$  and $q_k(\varphi_n-\varphi_m)\rightarrow0$ then  $\varphi \in D(q_k)$ and $q_k(\varphi_n-\varphi)\rightarrow 0$.
\end{lem}
\begin{proof}
Since  $q_k(\varphi_n-\varphi_m)\rightarrow0$ then $(V^{1/2}\varphi_n)_n$ and $(T_j\varphi_n)_n$ are Cauchy sequences in $ L^2(\mathbb{R}^d,w(x)dx)$
and so are convergent. Let $g_j=\lim T_j\varphi_n$ and $h=\lim V^{1/2}\varphi_n$. For any function $\psi\in C_0^\infty$,
\begin{eqnarray*}
 \langle g_j,\psi\rangle=\lim\langle T_j \varphi_n, \psi\rangle=-\lim\langle  \varphi_n,  T_j\psi\rangle=-\langle \varphi,T_j \psi\rangle.
\end{eqnarray*}
This yields that $T_j\varphi=g_j \in L^2(\mathbb{R}^d,w(x)dx)$. Similarly,
\begin{eqnarray*}
 \langle h,\psi\rangle=\lim\langle V^{1/2}\varphi_n, \psi\rangle=\lim\langle  \varphi_n,  V^{1/2}\psi\rangle=\langle \varphi,V^{1/2}\psi\rangle
=\langle V^{1/2}\varphi,\psi\rangle
\end{eqnarray*}
and so $ V^{1/2}\varphi=h\in L^2(\mathbb{R}^d,w(x)dx)$. Therefore $\varphi\in D(q_k)$ and $q_k(\varphi_n-\varphi)\rightarrow 0$
\end{proof}
Let $B_{q_k}$ be the associated sesquilinear form on $D(q_k)$. It follows that $D(q_k)$  is a Hilbert space with the inner product:
$$ \langle\varphi,\psi\rangle_{q_k}=\langle \varphi,\psi\rangle + B_{q_k}(\varphi,\psi),\quad \varphi,\;\psi\;\in D(q_k).$$
 The corresponding norm is given by
   $$\|\varphi\|_{q_k}= \sqrt{\|\varphi\|_{2,\;k}^2+ q_k(\varphi) }. $$
 The important consequence  of the above lemma is that  there exist  a unique positive self adjoint operator $L_k$ such that.
   $$q_k(\varphi)=\langle L_k(\varphi),\varphi \rangle,\qquad \varphi \in D(q_k),$$
and  is defined as  follows
   \begin{eqnarray}\label{def}
   D(L_k)  &=& \{\varphi\in D(q_k)/ \exists\; \widetilde{\varphi}\in L^2(\mathbb{R}^d,w(x)dx),\; B_{q_k}( \varphi,\psi)= \langle\tilde{\varphi}, \psi\rangle \;\forall\;\psi\in D(q_k)\}\nonumber
  \\    L_k(\varphi)& =& \widetilde{\varphi},\qquad \forall\;\varphi\in D(L_k).
\end{eqnarray}
 Moreover,
$$D(q_k)=D(L_k^{1/2})\quad\text{and}\quad q_k(\varphi)=\|L_k^{1/2}(\varphi)\|_{2,k}.$$
  \begin{thm}\label{Th2}
Assume that $V\in L^1_{loc}(\mathbb{R}^d,w_k(x)dx) $ and $V\geq0$, then  $C_0^\infty(\mathbb{R}^d)$ is dense in $D(q_k)$ in the norm $\|\varphi\|_{q_k}= \sqrt{\|\varphi\|_{2,\;k}^2+ q_k(\varphi) }.$
\end{thm}
For the proof we will need   the following lemmas.
\begin{lem}\label{L1}
The range of $e^{-L_k}$ is dense in the Hilbert space $( D(q_k), \|.\|_{q_k})$.
\end{lem}
\begin{proof}
    Let $\varphi\in D(q_k)$ such that
$$\langle e^{-L_k}\psi, \varphi\rangle+B_{q_k}( e^{-L_k}\psi, \varphi)=\langle e^{-L_k}\psi, \varphi\rangle+ \langle L_k e^{-L_k}\psi, \varphi\rangle=0 ,\quad \forall \psi\in L^2(\mathbb{R}^d,w_k(x)dx)$$
This implies that
$$\langle(L_k+1) e^{-L_k}\varphi,\psi\rangle =0,\quad \forall \psi\in L^2(\mathbb{R}^d,w_k(x)dx)$$
Since $L_k+1$ is  invertible, then   $e^{-L_k}\varphi=0$ which implies that $\varphi=0$. The density of $Ran(e^{-L_k})$  follows.
\end{proof}
\begin{lem}\label{L2}
$L^\infty \cap D(q_k)$  is dense in the Hilbert space $( D(q_k), \|.\|_{q_k})$.
\end{lem}
 \begin{proof}
 Using the Kato's stong Trotter product formula, see  Theorem S.21 of \cite{RS}, we have in the  strong convergence
\begin{equation}\label{trotter}
  s-\lim \left(e^{-tA_k/n} e^{-tV/n}\right)^n=e^{-tL_k}, \quad t\geq 0.
\end{equation}
By the fact that $|e^{-tV/n}f|\leq |f| $ and   $|e^{-tA_k}(f)|\leq e^{-tA_k}(|f|)$, which can be seen  from   (\ref{1}), it follows that
\begin{equation}\label{trott}
 |e^{-tL_k}(f)|\leq e^{-tA_k}(|f|).
\end{equation}
In particular we have
\begin{equation}\label{infty}
\|e^{-L_k}(f)\|_{\infty}\leq \| e^{-A_k}(|f|)\|_{\infty}\leq c \|f\|_{2,k}.,
\end{equation}
where the second inequality  follows from (\ref{1}), by using  Cauchy-Schwarz Inequality and (\ref{trp}).
From (\ref{infty}) we have  that
\begin{equation}\label{in1}
 Ran(e^{-L_k})\subset L^\infty.
\end{equation}
 Since the function $\lambda\rightarrow \lambda^{1/2}e^{-\lambda}$ is bounded on $(0,\infty)$ then by spectral theorem the operator  $L_k^{1/2}e^{-L_k}$ is bounded on $L^2(\mathbb{R}^d, w_k(x)dx)$.
We deduce that
\begin{equation}\label{dq}
 Ran(e^{-L_k})\subset D(L_k^{1/2}) =D(q_k).
\end{equation}
Similarly we have
$$Ran(e^{-L_k})\subset D(L_k)$$
 and   in view of   (\ref{in1}) and (\ref{dq}) we have that
$$Ran(e^{-L_k})\subset L^\infty \cap D(q_k).$$
We then conclude Lemma \ref{L2} from the result of Lemma \ref{L1}.
 \end{proof}
\begin{proof}[\textbf{Proof of Theorem \ref{Th2}}]
\par Let $$S=\{\varphi\in L^\infty \cap D(q_k);\;  supp(\varphi) \;is\; compact\}.$$
 We claim that $S$  is dense in  $( D(q_k), \|.\|_{q_k})$. Observe that for $\varphi\in L^\infty \cap D(q_k) $ and $\psi\in C_0^\infty(\mathbb{R}^d)$ be a radial function, we have that
$   V^{1/2}\varphi\psi \in  L^2(\mathbb{R}^d,w_k(x)dx)$ and
 in the distributional sense
  \begin{equation}\label{pp}
 T_j(\varphi\psi)   =T_j(\varphi) \psi +\varphi T_j(\psi) ,\quad j=1,2...d
  \end{equation}
   which gives that  $T_j(\varphi\psi)\in L^2(\mathbb{R}^d,w_k(x)dx) $. Hence  $\varphi\psi\in S$. From this fact if we choose a radial function  $\psi\in C_0^\infty(\mathbb{R}^d)$ with $\psi(x)=1$ near $0$ and we set $\varphi_n=\psi(./n)\varphi$ then  by Lebesgue dominated convergence theorem  and (\ref{pp}) the following hold
 \begin{itemize}
   \item $\|\varphi_n- \varphi\|_{2,k}\rightarrow 0$,
   \item $\|V^{1/2}\varphi_n- V^{1/2}\varphi\|_{2,k}\rightarrow 0$,
   \item    $\|T_j\varphi_n- T_j\varphi\|_{2,k} \rightarrow0$.
 \end{itemize}
This implies that $\varphi_n\rightarrow\varphi$ in norm $\|.\|_{q_k}$.
\par  We now  claim that $C_0^\infty(\mathbb{R}^d)$  is dense in $( D(q_k), \|.\|_{q_k})$.
Take  a radial function $\rho \in C_0^\infty(\mathbb{R}^d)$ with
$$\int_{\mathbb{R}^d}\rho(x)w_k(x)dx=1$$
For $\varphi\in S$ we  define $(\varphi_n)_n$ by $\varphi_n= \rho_n*_k\varphi$ where $\rho_n=n^{-2\gamma_k-d}\rho(x/n)$. Let us observe that $\varphi_n\in C_0^\infty(\mathbb{R}^d)$ and
 $$T_j\varphi_n=\rho_n*_kT_j\varphi,\quad j=1,2...d,$$
which can be  seen as  following: for $\psi\in  C_0^\infty(\mathbb{R}^d)$,
 \begin{eqnarray*}
 \int_{\mathbb{R}^d}\varphi_n T_j\psi &= & \int_{\mathbb{R}^d} \int_{\mathbb{R}^d}\varphi(y)\tau_x(\rho_n)(-y)T_j(\psi)(x)w_k(y)w_k(x)dydx
 \\&=& \int_{\mathbb{R}^d} \left(\int_{\mathbb{R}^d}\tau_{-y}(\rho_n)(x)T_j(\psi)(x)w_k(x)dx\right)\varphi(y)w_k(y)dy
 \\&=&- \int_{\mathbb{R}^d} \left(\int_{\mathbb{R}^d}\tau_{-y}(T_j(\rho_n))(x) \psi (x)w_k(x)dx\right)\varphi(y)w_k(y)dy
\\&=& -\int_{\mathbb{R}^d} \left(\int_{\mathbb{R}^d}\tau_{x}(T_j(\rho_n))(-y)\varphi(y)w_k(y)dy\right)\psi(x)w_k(x)dx
\\&=& -\int_{\mathbb{R}^d} \left(\int_{\mathbb{R}^d}T_j\tau_{x}(\rho_n)(-y)\varphi(y)w_k(y)dy\right)\psi(x)w_k(x)dx
\\&=&  \int_{\mathbb{R}^d} \left(\int_{\mathbb{R}^d}T_j\Big(\tau_{x}(\rho_n)(-\;.)\Big)(y)\varphi(y)w_k(y)dy\right)\psi(x)w_k(x)dx
\\&=& -\int_{\mathbb{R}^d} \left(\int_{\mathbb{R}^d}\tau_{x}(\rho_n)(-y)T_j(\varphi)(y)w_k(y)dy\right)\psi(x)w_k(x)dx
\\&=&- \int_{\mathbb{R}^d}\rho_n*_kT_j\varphi(x)\psi(x)w_k(x)dx
\end{eqnarray*}
 Therefore as convergent in $ L^2(\mathbb{R}^d,w_k(x)dx)$ we obtain
\begin{itemize}
   \item $\varphi_n\rightarrow \varphi$,
   \item $V^{1/2}\varphi_n\rightarrow V^{1/2}\varphi$,
   \item    $T_j\varphi_n\rightarrow T_j\varphi$
 \end{itemize}
 and  thus  $\varphi_n\rightarrow \varphi$ in the norm $\|.\|_{q_k}$. This conclude the proof of the density of $C_0^\infty(\mathbb{R}^d)$.
 \end{proof}
   Next we define   in the distributional way $ \mathcal{L}_{k,dist}= A_k+V$,  that is for
 $ \varphi \in  L^2(\mathbb{R}^d,w_k(x)dx)$,
 $$\int_{\mathbb{R}^d}  \mathcal{L}_{k,dist}\varphi(x)\psi(x)w_k(x)dx=\int_{\mathbb{R}^d} \varphi(x)(A_k +V)\psi(x)w_k(x)dx, \quad \forall \;\psi \in C_0^\infty(\mathbb{R}^d) .$$
 Clearly $ \mathcal{L}_{k,dist}= \mathcal{L}_{k}$ on $C_0^\infty(\mathbb{R}^d)$.
  \begin{cor}\label{cor}
 We have that  $$D(L_k)=\{ \varphi\in D(q_k);\;  \mathcal{L}_{k,dist}(\varphi)\in L^2(\mathbb{R}^d,w_k(x)dx)\}$$
  \end{cor}
 \begin{proof}
   For $\varphi\in D(q_k)$ and $\psi\in C_0^\infty(\mathbb{R}^d) $ we have
   \begin{eqnarray*}
    B_{q_k}(\varphi,\psi)&=&\sum_{j=1}^d \langle T_j(\varphi),T_j(\psi)\rangle + \langle  V\varphi,    \psi\rangle
    \\&=& -\sum_{j=1}^d \langle \varphi,T_j^2(\psi)\rangle + \langle  \varphi,   V \psi\rangle
    \\&=&\langle   \varphi, (A_k +V)\psi\rangle
     \\&=&\langle\mathcal{L}_{k,dist} \varphi, \psi\rangle
\end{eqnarray*}
We conclude the corollary by the definition (\ref{def}) of the domain $D(L_k)$ and the density of $C_0^\infty(\mathbb{R}^d) $.
 We add here that when $ \mathcal{L}_{k,dist}(\varphi)\in L^2(\mathbb{R}^d,w_k(x)dx)$,
\begin{equation}\label{55}
  H(\varphi)=\mathcal{L}_{k,dist}(\varphi)
\end{equation}
This fact will be used in the proof of the next theorem.
   \end{proof}
  \begin{thm}\label{th2}
Assume that $V\in L^2_{loc}(\mathbb{R}^d,w_k(x)dx) $ and $V\geq0$, then   $\mathcal{L}_k$ is essentially self-adjoint on $C^\infty_0(\mathbb{R}^d)$
and its closure  is $L_k$
\end{thm}
The proof of this  theorem  is the same as the proof of Theorem \ref{Th2}. Recall that $D\subset D(L_k)$ is a core of $L_k$ if for all
$\varphi\in D(L_k)$ there exist in $D$ a sequence $(\varphi_n)_n$ such that $\|\varphi_n-\varphi\|_{2,k}\rightarrow 0$ and
$\|L_k(\varphi_n)-L_k(\varphi)\|_{2,k}\rightarrow 0$.
\begin{lem}
  $D(L_k)\cap L^\infty$ is a core of $L_k$.
\end{lem}
\begin{proof}
Notice that we  have already proved that
\begin{equation}\label{12}
    Ran(e^{-tL_k})\subset D(L_k)\cap L^\infty.
\end{equation}
Let $\varphi\in D(L_k)$, by  the spectral theorem
$$\| e^{-tL_k}(\varphi) -\varphi\|_{2,k}^2=\int_0^\infty|e^{-t\lambda}-1|^2\;d(<P_\lambda \varphi ,\varphi>),$$
where $P_\lambda$ is the projection-valued measure  with respect to $L_k$. So using  the dominated convergence theorem
we obtain that $\| e^{-tL_k}(\varphi)-\varphi\|_{2,k}\rightarrow0$. Similarly,
$$\| L_ke^{-tL_k}(\varphi) -L_k(\varphi)\|_{2,k}= \| e^{-tL_k }L_k(\varphi)-L_k(\varphi)\|_{2,k}\rightarrow0.$$
Thus the lemma follows from (\ref{12}).
\end{proof}
\begin{proof}[\textbf{Proof of Theorm \ref{th2}}]
We first show for  $\varphi \in D(L_k)$ and $\psi\in C^\infty_0(\mathbb{R}^d)$ be a radial function we have that  $\varphi\psi\in D(L_k)$. Indeed, since $D(L_k)\subset D(q_k)$ then
 we already  have $\varphi\psi\in D(q_k)$ and   by a direct calculation
 \begin{eqnarray}\label{66}
\mathcal{L}_{k,dist}(\varphi \psi)&=&\mathcal{L}_{k,dist}(\varphi)\psi-2\sum_{j=1}^dT_j\varphi T_j\psi-\varphi\Delta_k\psi
\nonumber\\& +&
 \sum_{j=1}^d\sum_{\alpha\in R_+}k(\alpha) \alpha_j \frac{\Big(\varphi(x)-\varphi(\sigma_\alpha.x)\Big)\Big(T_j(\psi)(x)-T_j(\psi)(\sigma_\alpha.x)\Big)}{\langle x,\alpha\rangle}\nonumber
\\&&
 \end{eqnarray}
This is proven by showing that both sides have the same inner product with
a function in $C_0^\infty$ and  using the  density of $C_0^\infty$ in form norm.
 Since the function
  $$x\rightarrow \frac{ T_j(\psi)(x)-T_j(\psi)(\sigma_\alpha.x) }{\langle x,\alpha\rangle}$$
  is in $C^\infty_0(\mathbb{R}^d)$, it follows that  $\mathcal{L}_{k,dist}(\varphi \psi)\in L^2(\mathbb{R}^d,w_k(x)dx)$ and from  Corollary \ref{cor} we have   $\varphi\psi\in D(L_k)$.
\par Let $\psi\in C_0^\infty(\mathbb{R}^d)$ be  a radial function   with $\psi(x)=1$ near $0$ and  set $\varphi_n=\psi(./n)\varphi$.  In view of (\ref{55}) and  (\ref{66}) we get that
$\|H(\varphi_n)-H(\varphi)\|_{2,k}\rightarrow0$. This yields  that
   $$S'=\{\varphi\in D(L_k)\cap L^\infty/\; supp(\varphi) \;is\; compact \}$$
   is a core for $L_k$.
\par Now we proceed as follows. Let $\varphi\in S'$ then  $A_k\varphi+V\varphi\in L^2(\mathbb{R}^d,w_k(x)dx)$  and  $V\varphi\in  L^2(\mathbb{R}^d,w_k(x)dx)$,  since $V\in L^2_{loc}(\mathbb{R}^d,w_k(x)dx) $ and $\varphi\in L^\infty$. It follows that  $A_k\varphi\in L^2(\mathbb{R}^d,w_k(x)dx)$. However,
   if $\varphi_n= \rho_n*_k\varphi\in C^\infty_0(\mathbb{R}^d)$, where $(\rho_n)_n$ is defined in the proof of Theorem \ref{th1} and as  $\varphi \in L^\infty$ and $supp(\varphi)$ is compact then $\|V\varphi_n - V\varphi\|_{2,k}\rightarrow0$. But  by means of    Dunkl  transform we see that the
   $$A_k(\varphi_n)=\rho_n*_kA_k(\varphi)$$
 and  thus $\|A_k(\varphi_n) -A_k(\varphi)\|_{2,k}\rightarrow0$. This yields    that
$$\|\varphi_n - \varphi\|_{2,k}\rightarrow0,\quad \text{and}\quad \|\mathcal{L}_k(\rho_n)-A_k\varphi-V\varphi\|_{2,k}\rightarrow0$$
 and conclude  that $L_k$ is the closure of $\mathcal{L}_k$ on $C^\infty_0(\mathbb{R}^d)$.
  \end{proof}
 We closed this section  by    showing   that the semi-group corresponding to Schr\"{o}dinger  operator $L_k$    has
an integral kernel.
\begin{thm}\label{l1}
 $W_t= e^{-tL_k}$, $t>0$  is a kernel operator with
\begin{equation}\label{pettis}
  0\leq W_t(x,y) \leq  K_t(x,y)=\frac{1}{(2t)^{\gamma_k+d/2}c_k}e^{-(|x|^2+|y|^2)/4t} E_k(\frac{x}{\sqrt{2t}},\frac{y}{\sqrt{2t}}).
\end{equation}
  \end{thm}
 \begin{proof}
  Recall that from (\ref{trott})
  \begin{equation*}
   |W_t(f)|=|e^{-tL_k}(f)|\leq e^{-tA_k}(|f|),\quad f\in C_0^\infty(\mathbb{R}^d).
\end{equation*}
Thus  by Corollary \ref{3.2}, $W_t$ is a bounded operator from $L^p(\mathbb{R}^d,w_k(x)dx)$  to $L^\infty$ for all $1\leq p\leq \infty$.
The theorem of Dunford and
Pettis (see for example theorem 4.2 of \cite{AR}.) asserts  that such operator is a kernel operator.
Since $e^{tA_k}$ is an integral operator with positive kernel,
it is positivity preserving. Thus using the Trotter product formula (\ref{trotter}) we have that $W_t$ is positivity preserving   which implies  that  $W_t(x,y)\geq0$. The second inequality of (\ref{pettis}) follows from  (\ref{trott}) and from \cite{AR}, Theorems 2.2 and 4.3.
 \end{proof}
\section{Dunkl harmonic oscillator}
We first recall some known facts about the Dunkl harmonic oscillator.
The reader is referred to \cite{B1, B2, A, R1}.
\par The Dunkl harmonic oscillator is the Schr\"{o}dinger operator  $H_k =  - \Delta_k +|x|^2$.
It can be expressed in terms of  generalized Hermite functions $h_n^k$,
$$H_k(f)=\sum_{n\in \mathbb{N}^d}(2|n|+\gamma_k+d)\langle f, h_n^k\rangle \; h_n^k$$
where $|n|=n_1+...+n_d$. The functions  $h_n^k$ are eigenfunctions of $H_k$ with
$$H(h_n^k)=(2|n|+\gamma_k+d)\; h_n^k$$
and form an orthonormal basis of $L^2(\mathbb{R}^d,w_k(x)dx.$
\par The holomorphic Hermite semi-group  $e^{-zH_k}$, $Re(z) >0$ is given by
\begin{equation}\label{zH}
 e^{-zH_k}(f)=\sum_{n\in \mathbb{N}^d}e^{-z(2|n|+\gamma_k+d)}\langle f, h_n^k\rangle \; h_n^k.
\end{equation}
It has the following integral representation
$$e^{-zH_k}(f)(x)=\int_{\mathbb{R}^d}\mathcal{H}_z(x,y)f(y)w_k(y)dy,$$
where from the generalized Mehler-formula,
\begin{eqnarray*}
\mathcal{H}_z(x,y)&=&\sum_{n\in \mathbb{N}^d}e^{-z(2|n|+\gamma_k+d)}  h_n^k(x)  h_n^k(y) \\
 &=&c_k\left(\frac{\sinh(2z)}{2}\right)^{-\gamma_k-\frac{d}{2} }  E_k\left(\frac{x}{\sinh(2z)},y\right)\;e^{-\frac{\coth(2z)}{2}(|x|^2+|y|^2)}.
\end{eqnarray*}
 It can be written as
\begin{eqnarray*}
\mathcal{H}_z(x,y) =c_k \sinh(2z)^{-\gamma_k  } \int_{\mathbb{R}^d} \mathcal{H}_z^0(\eta,y) \; e^{-\frac{\coth(2z)}{2}(|x|^2-|\eta|^2)}d\nu_x(\eta).
\end{eqnarray*}
where $\mathcal{H}_z^0$ is the kernel of the classical Hermite semi-group given by
\begin{eqnarray}\label{ker1}\nonumber
\mathcal{H}_z^0(x,y)&=& (2\pi\sinh(2z))^{-\frac{d}{2}}\;e^{-\frac{\coth(2z)}{2}(|x-y|^2-\tanh(t)\;\langle x,y \rangle }
\\ &=&  (2\pi\sinh(2z))^{-\frac{d}{2}}\;e^{-\frac{1}{4} (\coth(z) (|x-y|^2+\tanh(z) |x-y|^2) }\;.
\end{eqnarray}
\begin{prop}\label{prop1}
For all $z\in \mathbb{C}$, $0\leq \arg(z)\leq \omega < \pi/2$ there exist $c>0$ and $C>0$ such that
$$|\mathcal{H}_z(x,y)|\leq  \mathcal{H}_{Re(z)}\;(cx,cy).$$
   \end{prop}
Let us first prove the following lemma
\begin{lem}\label{rez} If $0\leq \arg(z)\leq \omega <\pi/2$ then  there exist  $c>0$ and $C>0$ such that
\begin{eqnarray*}
  \quad c\;\coth(Re(z)) \leq Re(\coth(z))\leq C \;\coth(Re(z))
 \end{eqnarray*}
\end{lem}
\begin{proof}  Note first that for $z=t+iu$
$$Re(\coth(z))=\frac{e^{4t}-1}{(e^{2t}-1)^2+2e^{2t}(1-\cos(2u))}\leq \frac{e^{2t}+1}{e^{2t}-1}=\coth(t).$$
Now if $0\leq \arg(z)\leq \omega < \pi/2$, then
$$|u|\leq  \tan(\omega) \; t .$$
 Choosing $a= \pi/(4\tan(\omega)$, it follows that for $t\in(0,a]$, we have
$2|u|\leq  \pi/2$  and $\cos(2u) \geq \cos(2\tan(\omega) \; t)$. Then we get
 $$Re(\coth(z))\geq \frac{e^{4t}-1}{(e^{2t}-1)^2+2e^{2t}(1- \cos(2\tan(\omega)t)}. $$
If we take the function
$$\varphi(t)= \left(\frac{e^{4t}-1}{(e^{2t}-1)^2+2e^{2t}(1- \cos(2\tan(\omega)t)}\right)\; \left(\frac{e^{2t}-1}{e^{2t}+1}\right)$$
we see that
$$\lim_{t\rightarrow 0}\varphi(t)= 2\cos^2(\omega)>0$$
and   $\varphi$ define a positive   continuous function on the interval $[0,a]$, thus $\inf_{y\in (0,a]}\varphi(t)=c>0$.
Therefore, for $0<t\leq a$
 $$Re(\coth(z))\geq c\coth(t) .$$
For $t\geq a>0$ there exit $c>0$ so that
$$2e^{2t}(1-\cos(2u))\leq 4e^{2t}\leq c (e^{2t}-1)^2.$$
It follows that,
$$Re(\coth(z))\geq \frac{e^{4t}-1}{(1+c)(e^{2t}-1)^2}= c'\coth(t) .$$
The    lemma  follows.
\end{proof}
\begin{proof}[\textbf{Proof of Proposition \ref{prop1}}] This follows from (\ref{ker1}), Lemma \ref{rez} and the fact that
$$|\sinh(z)|= \sqrt{\sinh^2(t)+\sin^2(u)}\geq \sinh(t), $$
for $z=t+iu$. Indeed,
\begin{eqnarray*}
|\mathcal{H}_z^0(x,y)|& =& |  (2\pi\sinh(2z))^{-\frac{d}{2}}\;e^{-\frac{1}{4} (\coth(z) (|x-y|^2+\tanh(z) |x-y|^2) }|
\\&\leq &  (2\pi\sinh(2t))^{-\frac{d}{2}}\;e^{-\frac{c}{4} (\coth(t) (|x-y|^2+\tanh(t) |x-y|^2)}
\end{eqnarray*}
and
\begin{eqnarray*}
|\mathcal{H}_z(x,y)|& \leq &c_k \sinh(2t)^{-\gamma_k-d/2  } \int_{\mathbb{R}^d}  e^{-\frac{c}{4} (\coth(t) (|\eta-y|^2+\tanh(t) |\eta-y|^2 -\frac{\coth(2t)}{2}(|x|^2-|\eta|^2)}d\nu_x(\eta)\\
&=& c_k \sinh(2t)^{-\gamma_k-d/2  } \int_{\mathbb{R}^d}  e^{-\frac{c}{2} (\coth(2t) (|\eta-y|^2+
|x|^2-|\eta|^2) -2\tanh(t) \langle y,\eta\rangle)}d\nu_x(\eta)
\\
&=& c_k \sinh(2t)^{-\gamma_k-d/2  } e^{-\frac{c\coth(2t)}{2} (|x|^2+|y|^2)}
\int_{\mathbb{R}^d} e^{ \frac{c}{\sinh(2t)}\langle y,\eta\rangle}
 d\nu_x(\eta)
\\ &=&c_k \sinh(2t)^{-\gamma_k-d/2  } e^{-\frac{c\coth(2t)}{2} (|x|^2+|y|^2)} E_k( \frac{c}{\sinh(2t)}x,y)
\\&=& \mathcal{H}_{t}(\sqrt{c}\;x,\sqrt{c}\;y),
 \end{eqnarray*}
which is the desired inequality.
\end{proof}
\par Now, since from (\ref{pettis})
\begin{equation}\label{HK}
 0\leq \mathcal{H}_t(x,y)\leq K_t(x,y),\quad t>0,
\end{equation}
 then we can  state
\begin{cor}
  For all $z\in \mathbb{C}$, $0\leq \arg(z)\leq \omega < \pi/2$ there exist $c>0$ and $C>0$ such that
$$|\mathcal{H}_z(x,y)|\leq  K_{Re(z)}\;(cx,cy).$$
\end{cor}

\subsection{$H^\infty$-functional calculus on $L^p(\mathbb{R}^d, w_k(x)dx)$  for  Dunkl  oscillator  operator}
We briefly recall the definition of sectorial operators and their holomorphic
functional calculus. More on basic properties of sectorial operators
can be found in  \cite{ Mc, Maa, Mc1}.
 \par A closed operator $T$ on complex Hilbert space  is said to be sectorial   of type $\omega\in [0,\pi[$ if the following hold
\begin{itemize}
  \item[(i)] The spectrum $\sigma(T)\subset S_\omega=\{z\in \mathbb{C}^*,\; |Arg(z)|<\omega\}\cup\{0\}$
  \item[(ii)] For each $ \mu>\omega$  there exists $C_\mu$ such that $\|(T-zI)^{-1}\|\leq C_\mu|z|^{-1}$ for $z\notin S_\mu$.
\end{itemize}
A  non-negative  self-adjoint operator in a Hilbert space is an operator of type $S_\omega$ for all $\omega>0$.
 \\ Let $H^\infty(S_\mu^o)$ be the space of bounded holomorphic
functions in the open sector $S_\mu^o$, the interior of $S_\mu$  equipped with the
norm $\|f\|_{\infty}= \sup_{z \in S_\mu^o}|f(z)|$ and let
\begin{eqnarray*}
 \Psi(S_\mu^o)=\{\xi\in H^\infty(S_\mu^o);\;\exists\; s>0;\;  |\xi(z)|\leq C |z|^s/(1+|z|)^{2s} \}.
  \end{eqnarray*}
 For $\xi\in\Psi(S_\mu^o)$ we define the operator $\xi(T)$
\begin{equation}\label{0}
  \xi(T)=\frac{1}{2 \pi i}\int_{\gamma}\xi(z)(T-zI)^{-1}dz
\end{equation}
where $\gamma$ is the unbounded contour
\begin{equation}\label{}
  \gamma(t)=\left\{
              \begin{array}{ll}
                te^{i\theta}, & \hbox{$t\geq 0$}\\ -te^{-i\theta}, & \hbox{$t\leq 0$}
              \end{array}
            \right.
\end{equation}
 \par When  T is a one-one operator of type $\omega$  then one can  define  a   holomorphic functional calculus  as follows:
  Let $\psi $ the function   defined on $\mathbb{C}\setminus\{-1\}$ by $\psi(z) = z/(1+z)^2$. For each $\mu > \omega$ and for each  $f\in H^\infty(S_\mu^o)$ we have that $\psi$, $f\psi\in \Psi(S_\mu^o)$ and $\psi(T) $ is one-one. So $f\psi(T)$ is a bounded operator
   and $\psi(T)^{-1}$ is a closed operator. Define $f(T)$  by
\begin{equation}\label{00}
f(T)=\psi(T)^{-1}(f\psi)(T)=(I+T)^2T^{-1}(f\psi)(T).
\end{equation}
The definitions given by (\ref{0}) and (\ref{00}) are  consistent with the usual definition of polynomials of an operator.
\par We say that $T$ has   bounded  $H^\infty$-functional  calculus  if   further for all $f\in H^\infty(S_\mu^o)$  the operator $f(T)$ is bounded   and
$$\|f(T)\|\leq c_\mu\|f\|_\infty,$$
for some constant $c_\mu$.   An interesting result is given by the following
\begin{prop}[see \cite{Mc}, p.101]
If $T$ is a  positive self-adjoint
operator  then it    has a  bounded $H^\infty$- functional calculus for all $\mu > 0$ and
$$\|\xi(T)\|\leq  \|\xi\|_\infty,$$
for all $\xi\in H^\infty(S_\mu^o)$.
\end{prop}
\begin{thm}[  Convergence theorem, Th. D of \cite{Mc} ]\label{cvth}
Let T be a one-one operator of type $\omega$ on a Hilbert space and   $\{\xi_s \}$ be
a uniformly bounded sequence in  $H^\infty(S_\mu^o)$, $\mu>\omega$, which converges to a function  $\xi\in H^\infty(S_\mu^o)$  uniformly
on compact subsets of  $S_\mu^o$, such that $\{\xi_s(T)\}$ is a uniformly bounded set in  $\mathcal{L}(H)$.
Then $\xi(T)\in \mathcal{L}(H)$,  $\xi_s(T)(u)\rightarrow \xi(T)(u)$ for all $u\in H $, and 
$\|\xi(T)\|\leq \sup_{s}\|\xi_s(T)\|$.
\end{thm}
 \par  Now  Assume that an operator $T$ has a bounded $H^\infty$-functional calculus on the Hilbert space $L^2(\mathbb{R}^d,w_k(x)dx)$.  We say that $T$   has a
bounded $H^\infty$-functional calculus on $L^p(\mathbb{R}^d,w_k(x)dx)$  for $1<p<\infty$, if  for all $\xi\in H^\infty(S_\mu^o)$  the operator $\xi(T)$ can  be extended to a bounded operator on $L^p(\mathbb{R}^d,w_k(x)dx)$  that is,
$$\|\xi(T)(u)\|_p \leq c \|u\|_p,$$
for all $u\in L^p(\mathbb{R}^d,w_k(x)dx) $.
 \par  We conclude   with   the following  remark.
\begin{rem}
 A function $\xi\in H^\infty(S_\mu^o) $  is the limit of a uniformly bounded sequence of functions in $  \Psi(S_\mu^o)$
 in the sense of uniform convergence on compact subsets of $S_\mu^o$.
\end{rem}
\subsection{The main result }
  Our main result in this section is the following theorem.
\begin{thm}\label{mth}
The Dunkl harmonic oscillator  operator   $H_k$   has a
bounded $H^\infty$-functional calculus on $L^p(\mathbb{R}^d,w_k(x)dx)$  for $1<p<\infty$.  Moreover,
for each  $\xi\in  H^\infty(S_\mu^0)$, $0<\mu<\pi$, the operator $\xi(L_k)$ is of weak type
(1, 1).
\end{thm}
From this theorem we recover and extend the result of   \cite{B2} where a particular case is dealt with $\xi(z)=z^{ia}, a\in \mathbb{R}$. For the proof we follow the elegant approach of \cite{DR}.
\par  Let $k_t$ and $K_t$ the kernels given by (\ref{kk}) and (\ref{KK}). Recall that $$K_t(x,y)=\tau_x(k_t)(-y).$$
\begin{lem} \label{l2}
There exists $c>0$ such that for all $t>0$ and $z\in \mathbb{R}^d$,
\begin{equation}\label{}
    \sup_{y\in B(z,\sqrt{t})} K_t(x,y)\leq c \;\inf_{y\in B(z,\sqrt{t})}K_{2t}(x,y)
\end{equation}
   \end{lem}
\begin{proof}
  Recall that
   $$\tau_x(k_t)(-y)=t^{-\gamma_k-d/2}\int_{\mathbb{R}^d}e^{-(|y-\eta|^2+|x|^2-|\eta|^2)/t}d\nu_x(\eta)$$
Let $y_1,y_2\in B(z,\sqrt{t})$ and $\eta \in \mathbb{R}^d$ we have
$$\frac{|y_2-\eta |^2}{2t}\leq \frac{|y_1-\eta |^2+|y_2-y_1|^2}{t}\leq \frac{|y_1-\eta |^2}{t}+4. $$
It follows that
$$t^{-\gamma_k+d/2}e^{-|y_1-\eta |^2/t}\leq e^{4}2^{\gamma_k+d/2}(2t)^{-\gamma_k-d/2}e^{-|y_2-\eta |^2/2t} .$$
Now, since we  have
$$e^{-( |x|^2-|\eta|^2)/t}\leq e^{-( |x|^2-|\eta|^2)/2t}$$
thus for all $y_1,y_2\in B(z,\sqrt{t})$
$$\tau_x(k_t)(-y_1)\leq c \tau_x(k_{2t})(-y_2),$$
which yields the desired inequality.
\end{proof}
For $f\in L^1_{loc}(\mathbb{R}^d,w_k(x)dx)$,
the Dunkl maximal function $M_kf$ is defined by
 $$M_kf(x)=\sup_{r>0}\frac{1}{d_kr^{2\gamma_k+d}}|f*_k\chi_{B_r}|$$
where  $\chi_{B_r}$
is the characteristic function of the ball $B_r$ of radius $r$ centered at $0$.
According to the theorem 6.2 of \cite{XT}.
\begin{lem}\label{Mk}
There exists   $c>0$  such that
  $$\sup_{t>0}|k_t*f(x)|\leq c M_kf(x).$$
\end{lem}
\begin{lem}\label{lem1}
For all $\delta>0$ there exists a constant $c>0$ so that for all $r>0$ and We have that
$$\sup_{x\in \mathbb{R}^d}\int_{\min_{g\in G}| y-g.x| > r }K_t(x,y)w_k(y)dy\leq c(1+r^2t^{-1})^{-\delta}.$$
\end{lem}
\begin{proof}
  Let us note that for all $\eta\in co(G.x)$
$$\min_{g\in G}| y-g.x|\leq |x|^2+|y|^2-2\langle y,\eta\rangle\leq \max_{g\in G}| y-g.x|.$$
 Then we have
\begin{eqnarray*}
K_t(x,y)=\tau_x(k_t)(-y)&=&t^{-\gamma_k-d/2}\int_{\mathbb{R}^d}e^{(|x|^2+|y|^2-2\langle y,\eta\rangle)/t}d\nu_x(\eta)
 \\&\leq & e^{(-\min_{g\in G}| y-g.x|^2)/2t}\tau _{x}(k_{2t})(-y)
\end{eqnarray*}
and
\begin{eqnarray*}
 \int_{\min_{g\in G}| y-g.x| > r }K_t(x,y)w_k(y)dy&\leq& c e^{-r^2/t}\|\tau _{x}(k_{2t})\|_{1,k}= c \|k_{2t}\|_{1,k}e^{-r^2/t}=ce^{-r^2/t}
 \\&\leq & c (1+r^2t^{-1})^{-\delta}
 \end{eqnarray*}
 which is the desired inequality.
\end{proof}
\begin{lem}\label{wt}
  Let T be a bounded operator on $L^2(\mathbb{R}^d,w_k(x)dx )$ . Suppose that $(T_n)_n$ is a sequence of
bounded operators satisfy the condition: there exists $c>0$ such that for each $f \in L^2(\mathbb{R}^d,w_k(x)dx )\cap L^1(\mathbb{R}^d,w_k(x)dx ) $   and $\lambda>0$
$$   \mu_k\{x\in \mathbb{R}^d,\;\;  T_n(f)(x)> \lambda\}\leq c\; \frac{\|f\|_{k,1}}{\lambda},\quad n\in \mathbb{N}$$
 where the measure $d\mu_x=w_k(x)dx$.
Assume that for each $f \in L^2(\mathbb{R}^d,w_k(x)dx )\cap L^1(\mathbb{R}^d,w_k(x)dx ) $ there is a subsequence $(T_{n_j})_j$  such that
 \begin{equation}\label{nj}
    T(f)(x)=\lim_{j\rightarrow\infty} T_{n_j}(f)(x); \quad a.e.\; x\in \mathbb{R}^d .
 \end{equation}
Then $T$ is of weak type $(1, 1)$, i.e., there exists $c>0$ such that
$$   \mu_k\{x\in \mathbb{R}^d,\;\;  T(f)(x)> \lambda\}\leq c\; \frac{\|f\|_{k,1}}{\lambda}$$
for each $f \in L^2(\mathbb{R}^d,w_k(x)dx )\cap L^1(\mathbb{R}^d,w_k(x)dx ) $   and $\lambda>0$.
\end{lem}
\begin{proof}
  Let $f \in L^2(\mathbb{R}^d,w_k(x)dx )\cap L^1(\mathbb{R}^d,w_k(x)dx ) $   and $\lambda>0$. Put
$$A_j=\{x\in \mathbb{R}^d,\;\;  T_{n_j}(f)(x)> \lambda\}$$
and
$$C_j=\bigcap_{\ell\geq j}A_{n_\ell}.$$
 Then, we have $C_j\subset C_{j+1}$.
From (\ref{nj}) we see that
$$\mu_k(\{x\in \mathbb{R}^d,\;\;  T(f)(x)> \lambda\})\leq \mu_k\left( \bigcup_{j}C_j \right)=\lim_{j\rightarrow\infty}\mu_k\left( C_j \right)$$
 But,
 $$\mu_k(C_j) \leq  \mu_k(A_{n_j}) \leq c\; \frac{\|f\|_{k,1}}{\lambda}.$$
 Therefore
$$\mu_k(\{x\in \mathbb{R}^d,\;\;  T(f)(x)> \lambda\})\leq c\; \frac{\|f\|_{k,1}}{\lambda}$$
which is  the required inequality.
\end{proof}
 The  next lemma  show that the Euclidian $\mathbb{R}^d$ endowed with measure $\mu_k$ is a space of homogeneous type.
  Recall that a metric measure  space $(X,d, \varrho)$  is  said to be  of   homogeneous type    if $\varrho$ is a doubling measure on $X$, i.e.,
there exists $ c> 0$ such that  such that for all $x_0\in \mathbb{R}^d$ and $r>0$,
$$\varrho(B(x_0,2r))\leq c \varrho(B(x_0,r)).$$
Where $B(x_0,r)=\{x\in X,\; d(x,x_0)<r\}$.
\begin{lem}
 $\mu_k$ is a doubling measure.
\end{lem}
 \begin{proof}
   It is enough to check that the weight $w_k$ belongs to a Muckenhaupt class $A_p$ for some $p > 1$.
Indeed, it is known (\cite{Stein}, Ch V, 6.5) that when $P$ is a polynomial on $R^d$  having degree $\ell$ then $|P|^a$
belongs to $A_p$ whenever $-1<\ell a<p-1$. Applying this fact to the polynomial $P_\alpha(x) =  \langle x,\alpha\rangle$
for $\alpha\in R^+$ and taking $p > 2N\gamma_k + 1 $ where $N$ is the cardinality of $R^+$, we see that  $|P_\alpha|^{2Nk(\alpha)}\in A_p$.
Then according to (\cite{Stein}, Ch V, 6.1) and the fact that
$$w_k=\left(\prod_{\alpha\in R^+}|P_\alpha|^{2Nk(\alpha)}\right)^{\frac{1}{N}}$$
we obtain, with this choice of $p$, that $w_k \in A_p$.
 \end{proof}
We are now able to prove Theorem \ref{mth}.
\begin{proof}[\textbf{Proof of the Theorem \ref{mth}}]  Let $\xi\in H^\infty(S_\mu^o)$, $\mu>0$.
it suffices by Marcinkiewicz interpolation and duality  to prove that $\xi(H_k)$ is a weak-type $(1,1)$. Before  beginning  we note first that in view of the convergence theorem \ref{cvth} and lemma \ref{wt} one can assume that $\xi\in \Psi(S_\mu^o)$.
  \par Let $f\in L^2(\mathbb{R}^d), w_k(x)dx)\cap L^1(\mathbb{R}^d), w_k(x)dx) $ and $\lambda>0$. From the Calderon-Zygmund decomposition, there exist
a functions g and $f_j$  and balls $B_j=B(x_j , r_j)$ such that
  \begin{itemize}
    \item[(i)] $f=g+h$ with $h= \sum_j f_j$,
     \item[(ii)] $\|g\|_\infty \leq c \lambda $,
     \item[(iii)]  $supp(f_j)\subset B_j$, and
 \item[(iv)] $\|f_j\|_{1,k}\leq c \mu_k(B_j)$,
 \item[(v)]  $\sum_j\mu_k(B_j)\leq \dfrac{c}{\lambda}\;\|f\|_{1,k}$,
        \item[(vi)]    Each point of $\mathbb{R}^d$ is contained in at most a finite number  $M$ of the
balls $B_j$.
\end{itemize}
Note that (iv) and (v) imply that $\|h\|_{1,k}\leq c\|f\|_{1,k}$ . Hence
\begin{equation}\label{g}
 \|g\|_{1,k}\leq (1+c)\|f\|_{1,k}.
\end{equation}
\par Let  $S_t=e^{-tH_k}$, we split $h= h_1+h_2$ with
$$h_1= \sum_j S_{t_j}(f_j); \quad h_2= h-h_1=\sum_j(I-S_{t_j})f_j$$
where $t_j=r_j^2$. Then
\begin{eqnarray}
 \nonumber\mu_k\{x;\; |\xi(H_k)(f)(x)|>\lambda\} &\leq& \mu_k\left\{x;\; |\xi(H_k)(g)(x)|>\frac{\lambda}{3}\right\}
+\mu_k\left\{x;\; |\xi(H_k)(h_1)(x)|>\frac{\lambda}{3}\right\}
 \\&&+\mu_k\left\{x;\; |\xi(H_k)(h_2)(x)|>\frac{\lambda}{3}\right\}\label{T}
 \end{eqnarray}
 For the first term of the right hand side of (\ref{T}) we use  the $L^2$-boundedness of $\xi(H_k) $,  (\ref{g}) and $(ii)$ to get
$$\mu_k\left\{x;\;  \xi(H_k)(g)>\frac{\lambda}{3}\right\}\leq \frac{9}{\lambda^2}\|\xi(H_k)(g)\|_{2,k}\leq \frac{c}{\lambda^2}\|g\|_{2,k}\leq
\frac{c}{\lambda^2}\|g\|_{1,k}\|g\|_{\infty}\leq \frac{c}{\lambda}\|f\|_{1,k}.$$
For the second term    we have
$$\mu_k\left\{x;\; |\xi(H_k)(h_1)(x)|>\frac{\lambda}{3}\right\}\leq \frac{9}{\lambda^2}\|\xi(H_k)(h_1)\|_{2,k}\leq \frac{9}{\lambda^2} \sum_j\| S_{t_j}(f_j)\|_{2,k}^2$$
Using (\ref{HK}) and Lemma \ref{l2} we get
\begin{eqnarray*}
|S_{t_j}(f_j)(x)|&\leq&   = \int_{\mathbb{R}^d}\tau_x(k_{t_j})(-y)|f_j(y)|w_k(y)dy
\\&\leq&  \sup_{y\in B(z,\sqrt{t_j})}\tau_x(k_{t_j})(-y)\|f_j\|_{1,k}
\\ &\leq& c  \lambda \mu(B_j)\inf_{y\in B(z,\sqrt{t_j})}\tau_x(k_{2t_j})(-y)
\\ &\leq& c\lambda\int_{\mathbb{R}^d}\tau_x(k_{2t_j})(-y) \chi_{B_j}w_k(y)dy
\end{eqnarray*}
 and in view of Lemma \ref{Mk} it follows that for any  $\varphi\in L^2(\mathbb{R}^d,w_k(x)dx)$,
$$|\langle S_{t_j}(f_j),\varphi\rangle_{k}|\leq c\lambda \langle k_{2t_j}*_k|\varphi|,  \chi_{B_j}\rangle_{k}\leq  c\lambda \langle M_k|\varphi|,  \chi_{B_j}\rangle_{k}.$$
Thus  the  $L^2$-boundedness of $M_k$ yield
\begin{eqnarray*}
\left\|\sum_jS_{t_j}(f_j)\right\|_{2,k}&=&\sup\left\{ \left|\left\langle \sum_jS_{t_j}(f_j),\varphi\right\rangle_{k}\right|,\; \|\varphi\|_{2,k}\leq 1\right\}
\\&\leq& c\lambda \sup\left\{\left \langle  M_k|\varphi|,\sum_j \chi_{B_j}\right\rangle_{k},\; \|\varphi\|_{2,k}\leq 1\right\}
\\&\leq& c\lambda\left\| \sum_j\chi_{B_j}\right\|_{2,k}.
\end{eqnarray*}
We now use properties (vi)    of the Calderon-Zygmund decomposition
  to obtain the estimate
$$ \left\| \sum_j\chi_{B_j}\right\|_{2,k}\leq  M \left(\sum_j\mu(B_j)\right)^{\frac{1}{2}}. $$
Thus in view (v)
$$\left\|\sum_jS_{t_j}(f_j)\right\|_{2,k}^2\leq c\lambda \|f\|_{1,k}$$
and we conclude  that
$$\mu_k\left\{x;\; |\xi(H_k)(h_1)(x)|>\frac{\lambda}{3}\right\}\leq \frac{c}{\lambda}\|f\|_{1,k}.$$
Now  consider the  third term of the right hand side of (\ref{T}). Putting
$$\mathcal{B}_j= \bigcup_{g\in G}g\big(B(y_j,2r_j)\Big)=\bigcup_{g\in G}B(gy_j,2r_j)$$
 and  write
 $$\mu_k\left\{x;\;  |\xi(H_k)(h_2)(x)|>\frac{\lambda}{3}\right\}\leq \sum_j\mu_k (\mathcal{B}_j)+\mu_k\left\{x\notin\bigcup_{ j}\mathcal{B}_j ;\;|\xi(H_k)(h_2)(x)|>\frac{\lambda}{3}\right\}.$$
 By the doubling property of the measure $\mu_k$ and $(iii)$ we have
 $$\sum_j\mu_k ( \mathcal{B}_j ) \leq |G|\sum_j\mu_k ( B(y_j,2r_j)) \leq |G|\sum_j\mu_k ( B_j)\leq  \frac{c}{\lambda}\|f\|_{1,k},$$
since $w_k$ is a $G$-invariant function. So it   remains  to prove that
\begin{equation}\label{bbb}
  \mu_k\left\{x\notin\bigcup_{ j}\mathcal{B}_j ;\;|\xi(H_k)(h_2)(x)|>\frac{\lambda}{3}\right\}\leq
\frac{c}{\lambda}\;\|f\|_{1,k}.
\end{equation}
As in \cite{DR}  define   $\xi_j(v)=\xi(v)(1-e^{-t_jv})$ and write
$$\xi(H_k)(h_2)=\sum_j\xi_j(H_k)f_j.$$
We  represent the operator $\xi_j(H_k)$ by
$$\xi_j(H_k)=\frac{1}{2i\pi}\int_{\gamma}\xi_j(v)(H_k-vI)^{-1}dv=\xi_j^+(H_k)+\xi_j^-(H_k)$$
where
$$\xi_j^{\pm}(H_k)=\frac{1}{2i\pi}\int_{\gamma^{\pm}}\xi_j(v)(H_k-vI)^{-1}dv$$
and  the contour  $\gamma=\gamma^+\cup \gamma^-$ with $\gamma^+(t)=te^{i\theta}$ for $t\geq 0$ ,   $\gamma^-(t)=-te^{i\theta}$ for  $t< 0$ and whith $0<\theta<\pi/2$. Consider $\xi_j^{+}(H_k)$, for  $ v\in \gamma^+$ we substitute
$$(H_k-vI)^{-1}=\int_{\Gamma^+} e^{vz}e^{-zH_k}dz$$
where the curve $\Gamma^+$  is defined by  $\Gamma^+=te^{i\beta}$ with $\pi/2-\theta<\beta<\pi/2$.  It follows that
$$\xi_j^{+}(H_k)= \int_{\Gamma^+} n^+(z)e^{-zH_k} dz,$$
where
$$n^+(z)=\int_{\gamma^+}\xi_j(v)e^{vz} dv.$$
 Here we have used Fubini's theorem to change the order of integration. Therefore we have
$$ $$
$\displaystyle{ \int_{x\in \mathcal{B}_j^c} |\xi_j^{+}(H_k)(f_j)(x)|w_k(x)dx}$
\begin{eqnarray*}
&\leq& c \int_0^{+\infty} \int_0^{+\infty}
|1-e^{-t_jv}|\;e^{-|v||z|\sigma}\\&&\left(\int_{\mathbb{R}^d}\int_{x\in \mathcal{B}_j^c}|\mathcal{H}_z(x,y)||f_j(y)|w_k(x)w_k(y)dxdy\right)d|v|d|z|
\end{eqnarray*}
where $v=|v|e^{i\theta}$ and $\sigma=\cos(\theta+\beta) $.
Let us noting that $x\in \mathcal{B}_j^c$ is equivalent to the condition
$$\min_{g\in G}|x-gy|>2r_j.$$
Then using Proposition \ref{prop1} and Lemma \ref{lem1}
\begin{eqnarray*}
 \int_{x\in\mathcal{B}_j^c}|\mathcal{H}_z(x,y)|w_k(x)dx &\leq&  \int_{\min_{g\in G}|x-gy|>2r_j}K_{Re(z)}(cx,cy)w_k(x)dx
\\&=&c^{2\gamma_k} \int_{\min_{g\in G}|x-g.cy|>2r_j}K_{Re(z)}(x,cy)w_k(x)dx
\\&\leq& C \left(1+\frac{r_j}{Re(z)}\right)^{-\delta}\leq C \left(1+\frac{r_j}{|z|}\right)^{-\delta},
\end{eqnarray*}
from which it follows that $$ $$
$\displaystyle{ \int_{x\in \mathcal{B}_j^c} |\xi_j^{+}(H_k)(f_j)(x)|w_k(x)dx}$
\begin{eqnarray*}
&\leq& C\|f_j\|_{1,k} \int_0^{+\infty} \int_0^{+\infty}
|1-e^{-t_jv}|\;e^{-|v||z|\sigma }\left(1+ r_j |z|^{-1} \right)^{-\delta} d|v|d|z|.
\end{eqnarray*}
This integral is treated in \cite{DR} by  splitting  it
into two parts, $I_1$ and $I_2$, corresponding to integration over  $t_j|v|>1$ and  $t_j|v|\leq1$ and gives
$$\int_{x\in \mathcal{B}_j^c} |\xi_j^{+}(H_k)(f_j)(x)|w_k(x)dx\leq C\|f_j\|_{1,k}. $$
 We also obtain similar estimates for $\xi_j^{-}(H_k)$ by the same argument. Therefore
\begin{eqnarray*}
 \mu_k\left\{x\notin\bigcup_{ j}\mathcal{B}_j ;\;|\xi(H_k)(h_2)(x)|>\frac{\lambda}{3}\right\}&\leq & \frac{3}{\lambda}
\sum_j\int_{x\notin \cup_i\mathcal{B}_i } |\xi_j (H_k)(f_j)(x)|w_k(x)dx
\\&\leq& \frac{3}{\lambda}\sum_j\int_{x\notin \mathcal{B}_j } |\xi_j (H_k)(f_j)(x)|w_k(x)dx\\
&\leq& c\sum_j\|f\|_{1,k}\leq \frac{c}{\lambda}\;\|f\|_{1,k}.
\end{eqnarray*}
 This archives the proof   of the weak type estimates $(1,1)$ for $\xi(H_k)$.
 \end{proof}

  \end{document}